\documentclass[11pt]{article}

\usepackage[a4paper, left=3cm,top=3cm]{geometry}
\usepackage{amsmath,amssymb,amsthm,amsfonts}
\usepackage{hyperref,graphicx,subfigure,mathtools,enumitem}
\usepackage{authblk}
\usepackage[onehalfspacing]{setspace}
\usepackage[dvipsnames]{xcolor}
\usepackage[utf8]{inputenc}
\usepackage[ruled, linesnumbered]{algorithm2e}
\usepackage{siunitx}

\newtheorem{theorem}{Theorem}[section]
\newtheorem{remark}[theorem]{Remark}

\newtheorem{condition}[theorem]{Condition}
\newtheorem{lemma}[theorem]{Lemma}

\newcommand\set[1]{{\{#1\}}}
\newcommand\abs[1]{\left\vert#1\right\vert}
\newcommand\norm[1]{{\left\Vert#1\right\Vert}}

\newcommand\inner[2]{[#1] \cdot [#2]}
\newcommand\snorm[1]{\Vert#1\Vert}
\newcommand\sabs[1]{{\lvert#1\rvert}}

\newcommand{\R}{\mathbb R}

\newcommand{\N}{\mathbb N}
\newcommand{\sph}{\mathbb S}

\DeclareMathOperator{\UUo}{\mathbf{W}}
\DeclareMathOperator{\VVo}{\mathbf{W}^\sharp}

\DeclareMathOperator{\Ko}{\mathbf{K}}
\DeclareMathOperator{\Wo}{\mathbf{W}}
\DeclareMathOperator{\Xo}{\mathbf{X}}
\DeclareMathOperator{\Mo}{\mathbf{M}}

\DeclareMathOperator{\Eo}{\mathbf{E}}
\DeclareMathOperator{\Po}{\mathbf{P}}
\DeclareMathOperator{\Qo}{\mathbf{Q}}
\DeclareMathOperator{\Ao}{\mathbf{A}}

\DeclareMathOperator{\fixT}{\mathbf K}                        
\DeclareMathOperator{\timereversal}{\boldsymbol{\Lambda}}

\DeclareMathOperator{\supp}{supp}
\DeclareMathOperator{\prox}{prox}
\DeclareMathOperator{\diam}{diam}
\DeclareMathOperator*{\argmin}{arg\,min}
\DeclareMathOperator{\id}{Id}

\newcommand{\Dinner}{\Omega}
\newcommand{\Douter}{D}

\newcommand{\ini}{f}
\newcommand{\fini}{\mathbf{f}}
\newcommand{\data}{g}

\newcommand{\reg}{\mathcal{R}}    
\newcommand{\tik}{\mathcal{T}}

\newcommand{\Hini}{\mathbb{X}(\Dinner)}                 
\newcommand{\Hfinal}{\mathbb{Y}(\Douter^c)}              
\newcommand{\Hu}{\mathbb{H}}

\newcommand{\Ad}{\mathcal{A}}                 
\newcommand{\Dd}{\mathcal{D}}

\newcommand{\f}{\mathbf{f}}
\newcommand{\vu}{\mathbf{u}}
\newcommand{\en}{E}
\newcommand{\la}{\lambda}
\newcommand{\sg}{\sigma}
\newcommand{\e}{w}

\numberwithin{equation}{section}
\allowdisplaybreaks

\title{Full Field Inversion of the Attenuated Wave Equation: Theory  and Numerical Inversion}

\date{May 28, 2024}

\author{Ngoc Do}

\affil{Department of Mathematics, Missouri State University \authorcr
Springfield, Missouri, USA\authorcr
 E-mail: \tt{ngocdo@MissouriState.edu} 
}

\author{Markus Haltmeier}

\affil{Department of Mathematics, University of Innsbruck\authorcr
 Technikerstrasse 13, 6020 Innsbruck, Austria
 \authorcr E-mail:  \texttt{markus.haltmeier@uibk.ac.at}
 }

\author{Richard Kowar}

\affil{Department of Mathematics, University of Innsbruck\authorcr
Technikerstrasse 13, 6020 Innsbruck, Austria
 \authorcr E-mail: \texttt{richard.kowar@uibk.ac.at}
 }

\author{Linh V. Nguyen}

\affil{Department of Mathematics, University of Idaho\authorcr
              875 Perimeter Dr, Moscow, ID 83844, USA\authorcr
               E-mail: \tt{lnguyen@uidaho.edu}
}

\author{Robert Nuster}

\affil{Department of Physics, University of Graz\authorcr
Universitaetsplatz 5, 8010 Graz,  Austria\authorcr
E-mail: \tt{ro.nuster@uni-graz.at}
}

\begin{document}

\maketitle

\begin{abstract}
Standard photoacoustic tomography (PAT) provides data that consist of time-dependent signals governed by the wave equation, which are measured on an observation surface. In contrast, the measured data from the recently invented full-field PAT is the Radon transform of the solution of the wave equation on a spatial domain at a single instant in time. While reconstruction using classical PAT data has been extensively studied, not much is known about the full-field PAT problem. In this paper, we study full-field photoacoustic tomography with spatially variable sound speed and spatially variable damping. In particular, we prove the uniqueness and stability of the associated single-time full-field wave inversion problem and develop algorithms for its numerical inversion using iterative and variational regularization methods. Numerical simulations are presented for both full-angle and limited-angle data cases.

\end{abstract}

\section{Introduction}

Photoacoustic tomography (PAT) is a hybrid tomographic imaging modality that combines the high resolution of ultrasound imaging with high optical contrast. In PAT, an object is irradiated by short laser pulses, which then produces an acoustic pressure wave as a result of light absorption inside the object. In the conventional approach, pressure waves are recorded as temporally resolved signals outside of the object and used to reconstruct the initial pressure. The associated wave inversion problem consists of recovering the initial data of the wave equation from data $u|_{\Gamma \times [0,T]}$ where $u$ is the induced pressure wave, $T$ is the measurement time, and $\Gamma$ is the measurement surface. For the standard setup, several theoretical and numerical results on uniqueness, stability, and reconstruction algorithms have been developed in various situations including variable speed of sound, acoustic damping, and limited data \cite{PAT_review2,poudel2019survey}. In contrast to the standard problem, in this work, we address the more recent  full-field PAT invented in \cite{ccd,fullfield_intro}, where the associated wave inversion problem consists of recovering the initial data of the wave equation from measurements $u(\cdot, T)$ at a single time instance.

\subsection{Full field PAT}
\label{sec:ffp}

In this article, we will consider full field PAT for the damped wave equation
\begin{align}  
\label{eq:wave1}
&[c^{-2}(x)\partial_{tt} + a(x)\partial_t - \Delta] u(t,x)
= 0 \quad (x,t)\in\R^d\times (0,\infty)
\\ \label{eq:wave2}
& u(x,0)   = \ini_1(x) \quad x \in\R^d
\\ \label{eq:wave3}
& \partial_t u(x,0) = \ini_2(x)  \quad x \in \R^d \,,
\end{align}
where $\fini = (\ini_1, \ini_2)$ are the initial data, $c$ is the spatially varying  speed of sound, and $a$ is the spatially varying damping coefficient. For  the application the cases $d=2,3$ are  relevant, but in the theory part we will consider general dimension. Note that even for standard PAT data, attenuation is still an ongoing research topic, and various attenuation models have been used and studied in \cite{la2006image,nguyen2016dissipative,homan,palacios2016reconstruction,treeby2010photoacoustic,kowar2011attenuation}. Here, for the sake of clarity, we will only work with    \eqref{eq:wave1}-\eqref{eq:wave3}. 

Measurements in full-field PAT consist of line integrals of the solution of \eqref{eq:wave1}-\eqref{eq:wave3} at a single measurement time $T>0$ outside the sample, from which we aim to recover the initial data $\fini = (\ini_1, \ini_2)$ assumed to be supported in a domain $\Dinner \subseteq \R^3$.
More formally, let us denote by $\Dinner^c = \R^3 \setminus \Dinner$ the part outside of $\Dinner$ and consider the initial-to-final time wave operator $\fini \mapsto \Wo_T \fini \coloneqq u(\cdot,T)|_{\Dinner^c}$ and the 2D X-ray transform $h \mapsto \Xo h$ defined by 
\begin{equation} \label{ew:X}
\Xo h (\theta,\xi,x_3) \coloneqq \int_\R h(\xi\theta+s\theta^\perp,x_3) \, ds 
\end{equation}
for $(\theta,\xi,x_3) \in \sph^1 \times \R \times \R$ such that $(\xi\theta+s\theta^\perp,x_3)$ is contained in $\Dinner^c$.
Note that restricting the set of line integrals even in the complete data case reflects the fact that in practice, at least integrals over lines intersecting $\Dinner$ are not available. More projection data might be missing due to other practical constraints.

With the above notation, the full-field PAT inversion problem consists of recovering the PA source $\fini = (\ini_1, \ini_2)$ from a noisy approximation of the data $\Xo \Wo_T \fini$. Image reconstruction can be performed in two steps by first inverting the X-ray projection $\Xo$ and then inverting the initial-to-final time wave operator $\Wo_T$. Inverting $\Xo$ is a well-studied problem. In this paper, we develop the theory for the inversion of $\Wo_T$.

\subsection{Main results}

As the main results of this paper, for the first time, we establish theoretical results for recovering $(\ini_1, \ini_2)$ from single-time full-field data $\Wo_T \ini = u(\cdot, T)|_{\Dinner^c}$ of the solution of \eqref{eq:wave1}-\eqref{eq:wave3}. Here $T > 0$ is the measurement time and $\Dinner \subseteq \R^d$ is the measurement domain. Specifically, we establish the uniqueness and stability of the inversion and provide an exact time reversal inversion method for it. These results extend the analysis of \cite{zangerl2019full,time_reversal_ff}, where we established uniqueness and stability for the variable sound speed case in the absence of attenuation.

Concerning image reconstruction, we go one step further and directly work with the full-field PAT data instead of first inverting the X-ray transform and the wave equation. In particular, we extend the one-step algorithm \cite{haltmeier2019}, in which the source is recovered directly from the full-field measured data. For that purpose, we apply iterative and variational regularization methods. We also investigate the influence of noise and limited data on the reconstruction of various methods.

\section{Theory}

Throughout this manuscript, let $\Omega \subseteq \R^d$ be a bounded domain where $d \in \N$ is the spatial dimension, and let $c, a \colon \R^3 \to \R$ be the variable sound speed and variable damping coefficient, respectively, where we assume that $a$ and $c-1$ are both smooth and supported inside $\Dinner$. We further assume $a$ is nonnegative and that $c$ is bounded away from zero. We consider the damped wave equation \eqref{eq:wave1}-\eqref{eq:wave3} with pairs of initial data $\fini = (\ini_1, \ini_2)$ in the Hilbert space $\Hini \coloneqq H_0(\Dinner) \times L^2(\Dinner, c^{-2})$ equipped with the energy norm $\norm{\fini}^2_{\Hini} = \norm{\nabla \ini_1}^2_{L^2(\Dinner)} + \norm{\ini_2}^2_{L^2(\Dinner, c^{-2})}$. 
Moreover, let $\Douter \subseteq \R^d$ be a domain such that $\supp(u(\cdot, t)) \subseteq \Douter$ for all initial data $\fini = (\ini_1, \ini_2) \in \Hini$ and all $t \in [0,T]$. Set $\Dinner^c \coloneqq \Douter \setminus \Dinner$ and consider $\Hfinal = \{ g \in H^1(\Dinner^c) \mid \supp(g) \subseteq \Douter \setminus \Dinner \}$ equipped with the norm $\norm{\data}_{\Hfinal}^2 = \norm{\nabla \data}^2_{L^2(\Dinner^c)}$.

For the mathematical  analysis we will study the forward operator 
\begin{equation} \label{eq:fwd}
\UUo_T \colon \Hini \to  \Hfinal \colon  \fini \mapsto  u(\cdot,T)|_{\Dinner^c} \,,
\end{equation}
where $u(\cdot,T)$ solves \eqref{eq:wave1}-\eqref{eq:wave3}  with initial data $\fini \in \Hini$. Standard theory of  second order hyperbolic equations \cite{evans2022partial} then gives existence and uniqueness  of a weak solution that in particular satisfies $u  \in C([0,T], H^1_0(\R^d))$, making the forward map $\UUo_T$ well defined and bounded.  In this section we establish  uniqueness and  stability of inverting  $\UUo_T$ and  present an exact time reversal  inversion technique.

\subsection{Uniqueness}

We begin the theoretical analysis with the uniqueness of the full wave inversion, which amounts to the injectivity of the forward operator defined in \eqref{eq:fwd}.

\begin{lemma}[Auxiliary uniqueness  results]\label{lem:uni} 
Let $u$ be the solution of  \eqref{eq:wave1}-\eqref{eq:wave3} with initial data $\fini \in \Hini$.
\begin{enumerate}[topsep=0em, itemsep=0em,label=(\alph*)]
\item \label{it:uni1} $u( \cdot ,T)|_{\Dinner^c} = 0 $ implies  $(\partial_t u)(\cdot,T)|_{\Dinner^c} = 0$.
\item  \label{it:uni2} If $T>\diam(\Dinner)$, then $u( \cdot ,T)|_{\Dinner^c} = (\partial_t u)(\cdot,T)|_{\Dinner^c} = 0$ implies  $\fini = 0$.
\end{enumerate}
\end{lemma}

\begin{proof}
Item \ref{it:uni1} has been established in  \cite{time_reversal_ff} for the case of vanishing attenuation. Using that the respective arguments  of \cite{time_reversal_ff}  are based on the exterior problem for the wave equation and that we  have $a|_{\Dinner^c} = 0$ gives the claim. Item \ref{it:uni2} follows from the argument of \cite{homan}  on the uniqueness with standard PAT data that. Note that the argumentation there is independent of th particular form of the initial data.    
\end{proof}

From Lemma~\ref{lem:uni}, we conclude the following  uniqueness result.

\begin{theorem}[Uniqueness of full field wave inversion]\label{T:uni} 
If  $T>\diam(\Dinner)$, then  the single time wave transform $\UUo_T$ defined by \eqref{eq:fwd} is an injective bounded linear mapping.
\end{theorem}

\begin{proof}
Due to the linearity  of  $\UUo_T$ it suffices to show that  $\UUo_T \fini = u(\cdot,T)|_{\Dinner^c}=0$ implies  $\fini =0$.  Because $T>\diam(\Dinner)$  this  follows by combining \ref{it:uni1}, \ref{it:uni2}  from Lemma \ref{lem:uni}.      \end{proof}

\subsection{Stability}

Next, we turn to the stability of inverting $\UUo_T$ in terms of the norms $\norm{\cdot}_\Hini$ and $\norm{\cdot}_\Hfinal$. For this, we will make use of the visibility Condition~\ref{cond:vis}, which in particular implies $T > \diam \Dinner$ and thus, due to Theorem~\ref{T:uni}, the injectivity of $\UUo_T$.

\begin{condition}[Visibility] \label{cond:vis}
All geodesics with respect to the metric $c^{-2}(x)\,dx^2$ that are strictly contained in $\Dinner$ have length less than $T$. 
\end{condition}

The following lemma is a major ingredient of our stability analysis. It can be derived from the results of \cite{palacios2016reconstruction,palaciospartial}, involving sophisticated microlocal analysis. Below, we will present a different and more elementary one.

\begin{lemma} \label{lemma:stability} 
Suppose  the visibility condition  holds, let $u$ be the solution of \eqref{eq:wave1}-\eqref{eq:wave3} with initial data $\fini \in \Hini$ and write $\vu = (u,\partial_t u)$. Then for some  constant $C$ independent of $\fini$, the following stability  estimate holds  
\begin{equation} \label{eq:stab}
\norm{\fini}_{\Hini} \leq C \norm{\vu(\cdot,T)}_{\Hu(\Dinner^c)}^2 \,,
\end{equation}
where $\norm{\vu(\cdot,T)}_{\Hu(\Dinner^c)}^2 \coloneqq \|\partial_t u\|_{L^2(\Dinner^c)}^2 + \|\nabla u\|_{L^2(\Dinner^c)}^2$ 
\end{lemma}

\begin{proof} Application  of $\partial_t u$    to \eqref{eq:wave1} and  integration over $\R^d$ yields 
$\frac{1}{2} \frac{d}{dt} \int_{\R^d} \bigl[ c^{-2}  \abs{\partial_t u(\cdot ,t)}^2 +  \abs{ \nabla u (\cdot ,t)}^2  + a  \abs{\partial_t u(\cdot,t)}^2 \bigr] =0
$.  With the constant $A \coloneqq  \max \set{ c^{2}(x) a(x) \mid x \in \R^d}$ we get 
$\frac{d}{dt} \int_{\R^d} \bigl[  c^{-2} \abs{\partial_t u(\cdot,t)}^2 +  \abs{ \nabla u (\cdot,t)}^2 +a  \abs{\partial_t u(\cdot,t)}^2 \bigr]   \geq 0$, and thus
\begin{equation*}  \forall t \in [0,T] \colon \quad 
 \frac{d}{dt} \int_{\R^d} e^{2At} \bigl[c^{-2}  \abs{\partial_t u(\cdot,t)}^2 +  \abs{ \nabla u (\cdot,t)}^2\bigr]  \geq 0 \,. 
 \end{equation*}
Integration over $[0,T]$, using the  initial conditions  \eqref{eq:wave2}, \eqref{eq:wave3} and the definition of the norm $\norm{\cdot}_{\Hini}$ yield        
\begin{equation} \label{eq:stab-aux} 
\norm{\fini}^2_{\Hini} \leq e^{2AT} \int_{\R^n} \bigl[c^{-2}  \abs{\partial_t u(\cdot,T)}^2 +  \abs{ \nabla u (\cdot,T)}^2\bigr]     \,. 
\end{equation}
Next  consider the operator $\VVo_T \colon \Hini \to \Hini$  defined by $ \VVo_T \fini \coloneqq ( u(\cdot,T)|_\Omega - \phi, \partial_t u(\cdot,T)|_\Omega)$, where $\phi$ is the harmonic extension of $u(\cdot,T)|_{\partial \Dinner}$ to $\Omega$.  Under the visiblility condition, the function $\VVo_T(\fini)$ is smooth for all $\fini \in \Hini$ and thus $\VVo_T$ is a compact operator.  Therefore, the space $V$ defined as the set of all $\fini \in  \Hini$ with  $\norm{ \VVo_T(\fini) }_{\Hini} \geq e^{-2AT} \norm{ \fini }_{\Hini} $ has finite dimension.  

Using  \eqref{eq:stab-aux} and the trace inequality, for all $\fini \in V^\perp$,  we get    
\begin{align*}
\norm{\fini}^2_{\Hini} 
&\leq e^{2AT} \bigl(\norm{ \VVo_T(\fini) }^2_{\Hini} + \|\nabla \phi\|^2_{L^2(\Dinner)} + \norm{\vu(\cdot,T)}_{\Hu(\Dinner^c)}^2 \bigr) 
\\
& \leq e^{-2AT} \norm{\fini}^2_{\Hini}  + e^{2AT} \bigl(\|\nabla \phi\|^2_{L^2(\Dinner)} + \norm{\vu(\cdot,T)}^2_{\Hu(\Dinner^c)}\bigr) 
\\
& \leq e^{-2AT} \norm{\fini}^2_{\Hini}  + e^{2AT} \bigl( C \norm{\vu(\cdot,T)}_{\Hu(\Dinner^c)}^2 + \norm{\vu(\cdot,T)}^2_{\Hu(\Dinner^c)}\bigr) \,, 
\end{align*}
for some generic constant $C>0$, which implies 
\begin{equation} \label{E:Vperp} 
\forall \fini \in V^\perp \colon \quad  
\norm{\fini}^2_{\Hini} \leq  \frac{(C+1)e^{2AT}}{1-e^{-2AT}} \norm{\vu(\cdot,T)}_{\Hu(\Dinner^c)}^2   \,.
\end{equation}
Since $\UUo_T$  is injective on $V$ and $V$ is  finite dimensional space,  we have $\forall \fini \in V \colon \norm{\fini}^2_{\Hini} \leq C\norm{\vu(\cdot,T)}^2_{\Hu(\Dinner^c)}$. Together with \eqref{E:Vperp} this shows \eqref{eq:stab}.
\end{proof}

\begin{lemma} \label{lemma:energy} In the situation of Lemma \ref{lemma:stability} we have $ \norm{\vu(\cdot,T)}_{\Hu(\Dinner^c)} \leq C \|u(\cdot,T)\|_{H^1(\Dinner^c)}$.
\end{lemma}

\begin{proof}
Let $\data = u(\cdot,T)|_{\Dinner^c}$ and  $h = \partial_t u(\cdot,T)|_{\Dinner^c}$. We will prove that microlocally $h = \Ao(\data)$, where $\Ao$ is a pseudo-differential operator of order $1$. Indeed, let  $(x_0,\xi_0) \in T^*\Dinner^c$. Then near $(x_0,T)$ (see \cite{SUBrain}) we  have   
\begin{equation*}
	u(x,t) \simeq (2\pi)^{-3} \sum_{\sg=\pm} \int_{\R^d} e^{i\phi_\sg(x,t,\xi)} [a_{1,\sg}(x,t,\xi) \hat{g}(\xi) + a_{2,\sg}(x,t,\xi) \hat{h}(\xi)] d\xi = \sum_{\sigma} u_\sigma(x) \,,
\end{equation*}
 where $\phi_{\sg}(x,T,\xi) =  x \cdot \xi$, $a_{1,\sg}(x,T,\xi) = 1/2$ and $ a_{2,\sg}(x,T,\xi) = \sg/(2 c(x) |\xi|)$. 
Let $(x(t), \xi(t))$ be the bicharacteristic passing through $(x_0,\xi_0) \in T^* \Dinner$ at $t = T$, where $x(t)$ is a unit speed geodesics in the metric $c^{-2}(x)\,dx^2$. Further denote $(x_+,\xi_+) = (x(0),\xi(0))$ and  $(x_-,\xi_-)  = (x(2T),\xi(2T))$. Then, the singularity of $u_\sigma$ at $(x_0,\xi_0)$, if exists, is generated by that of $\fini = (\ini_1, \ini_2)$ at $(x_\sigma,\xi_\sigma)$. Since the geodesic distance between $x_+$ and $x_-$ is $2T$, at least one of them is outside of $\Dinner$. Without loss of generality, we assume that $x_- \notin \Dinner$. Then $\fini = 0$ near $x_-$. Therefore, $u_- \simeq 0$ along the bicharacteristic $(x(t),\xi(t))$. In particular, for $t=T$, we have   $ \int_{\R^d} e^{i x\cdot \xi } [\hat{g}(\xi) - \hat{h}(\xi)/(c(x) |\xi|) ] d\xi \simeq 0 $
or  $ g(x) \simeq (2\pi)^{-3} \int_{\R^d} e^{i x\cdot \xi }   \hat{h}(\xi)/(c(x) |\xi|) d\xi$.
By inverting the above elliptic pseudo-differential operator, we have, up to lower order terms, $	h(x) \simeq (2\pi)^{-3} \int_{\R^d} e^{i x\cdot \xi } c(x) |\xi| \hat{g}(\xi) d\xi$.
That is, microlocally on $\Dinner^c$, $h = \Ao(h)$ where   $\Ao$ is a pseudo-differential operator of order $1$ and therefore  
$\|h\|_{L^2(\Dinner^c)} \leq C\|g\|_{H^1(\Dinner^c)} \leq C\norm{\data}_{\Hfinal}$.
 \end{proof}

Combining Lemmas~\ref{lemma:stability} and ~\ref{lemma:energy} we obtain the following  stability result.

\begin{theorem}[Stability of full field wave inversion]\label{thm:stability} 
If the visibility Condition~\ref{cond:vis} holds, then there exists a constant $C>0$, such that  
\begin{equation} \label{eq:stability} 
\forall \fini  \in \Hini \colon \quad 
\norm{\fini}_{\Hini} \leq C \norm{\UUo_T(\fini)}_{\Hfinal} \,.
\end{equation}
\end{theorem}

\begin{proof}
Let $u$ be the solution of \eqref{eq:wave1}-\eqref{eq:wave3} with initial data $\fini \in \Hini$. Then $\UUo_T(\fini) = u(\cdot, T)$ and combining Lemmas~\ref{lemma:stability} and ~\ref{lemma:energy}  gives estimate \eqref{eq:stability}.  
\end{proof}

\subsection{Iterative time-reversal}

Next we consider  the iterative  solution of the full field PAT problem  $ \data = \UUo_T \fini$ based on the time reversed wave equation
\begin{align}  
\label{eq:timerev1}
& [ c^{-2}(x)\partial_t^2 - a(x)\partial_t-\Delta] v\,(x,t) = 0 \quad (x,t)\in\R^d\times (0,\infty)
\\ \label{eq:timerev2}
& v(x,T)   = h(x) \quad x \in\R^d
\\ \label{eq:timerev3}
& \partial_t v(x,T) = 0  \quad x \in \R^d \,  
\end{align} 
and the associated time reversal  $\timereversal_T h \coloneqq  (v(\cdot,0), \partial_t v(\cdot,0))$.  
Let us further denote by      
\begin{align*}
&\Eo \colon \Hfinal \to  H^1(\Douter)
\\
& \Po \times \Qo \colon H^1_0(\R^d) \times L^2(\R^d,c^{-2}) \to \Hini
\end{align*}
denote the extension operator from $\Hfinal = H^1(\Dinner^c)$ to $H^1_0(\Douter)$ and  the  orthogonal projection of $H^1_0(\Dinner) \times L^2(\Omega, c^{-2})$  onto $\Hini = H^1_0(\Dinner) \times L^2(\Omega, c^{-2})$, respectively. 

\begin{remark}[Extension and projection operators]
The extension operator $\Eo$ is require in order to obtain  proper initial data   for the time reversed equation based on  elements in data space $\Hfinal$ whereas $\Po$ is required to map the time reverses back to the space where the initial data lives.  The extension operator applied to $g\in \Hfinal$ is given by  $\Eo(g)|_{\Omega^c} = g$  and $\Eo(g)|_{\Omega}=\phi $ where $\phi$ satisfies the  Dirichlet problem
\begin{equation}\label{eq:laplace}
\left\{ \begin{aligned}
    &\Delta \phi = 0   && \text { in }   \Omega \\ 
     &\phi =  g|_{\partial \Omega} && \text { on }   \partial \Omega \,.
\end{aligned}
\right. 
\end{equation}
Here, $g |_{\partial \Omega} \in H^{1/2} (\partial \Omega) $ denotes the trace of $g \in H^1(\Omega^c)$ on $\partial \Omega$.  Note that the Dirichlet interior problem \eqref{eq:laplace} has a unique solution $\phi \in H^1(\Omega)$ (see, for example, \cite{mclean2000strongly}).  Further, the orthogonal projection $\Po \times \Qo$  is  given by   $(\Po \times \Qo)(g, h) =  (\Po h, \Qo h) = (g|_\Omega - \phi, h|_\Omega)$, where $\phi \in H^1(\Omega)$ is the solution of  \eqref{eq:laplace}.   
\end{remark}

Iterative time reversal is then defined  as  Neumann series   $ \sum_{j \in \N} (\id - \la \VVo_T \UUo_T)^j$ where  
\begin{equation} \label{eq:timereversal}
	\VVo_T  \coloneqq (\Po \times \Qo )   \timereversal_T    \Eo \colon  \Hfinal \to \Hini \,.
\end{equation}
The key of iterative time reversal is to show that $\norm{\fixT_\la} < 1 $, which we will show next.

\begin{theorem} \label{thm:timerev}
Suppose $T > T_0/2 $ and consider for any $\lambda \in (0,2]$ the operator $\fixT_\la  \coloneqq \id - \la \VVo_T \UUo_T\colon  \Hini \to \Hini $.  Then the following hold: 
\begin{enumerate}[topsep=0em, itemsep=0em,label=(\alph*)] 
\item  \label{timerev1} $\fixT_2$ satisfies $\forall \fini \in \Hini \setminus\set{0}  
              \colon \|\fixT_2 \f\|_{\Hini} < \|f\|_{\Hini}$. 
\item  \label{timerev2} If $\lambda \in (0,2)$, then  $\|\fixT_\la\| < 1$. 
\end{enumerate}
\end{theorem}

\begin{proof} \mbox{}
\ref{timerev1}: Let $v$ solve the time-reversed wave equation   \eqref{eq:timerev1}-\eqref{eq:timerev3} with  initial data $h = 2 \cdot (\bar g)$ where $\bar g \coloneqq \Eo g$. Then $w \coloneqq u - v$ satisfies the wave equation 
$c^{-2}\,\partial_t^2 w = - \Delta w - a\,\partial_t (u+v)$ and the corresponding energies at times $0$ and $T$ 
respectively satisfy  
\begin{align}
   \en_w(0) &= 
       \int_{\R^n}  \left[ c^{-2}(x)\,| \ini_2(x) - \partial_t v (x,0)|^2
              + \sabs{ \nabla \ini_1(x) - \nabla  v(x,0)  }^2 \right] \,\mathrm{d} x,  \nonumber \\  \label{eq:ene2}
   \en_w(T) &= 
       \int_{\R^n}  \left[ c^{-2}(x)\,|\partial_t u (x,T)|^2 
              + \sabs{ \nabla g (x) - 2 \nabla \bar g (x) }^2 \right] \,\mathrm{d} x  \,.
\end{align}
The trace extension  satsfies   $\bar g |_{\partial \Dinner} = g |_{\partial \Dinner} $ and 
$(\Delta \bar g) |_{\Dinner} =0$. Therefore   
\begin{equation*}
       \int_{\Dinner} (\sabs{2 \nabla \bar g - \nabla g }^2 - \abs{\nabla g}^2) \, \mathrm{d} x
          = 4  \int_{\Dinner} \inner{\nabla \bar g  }{ \nabla ( \bar g- g) } \,\mathrm{d} x
           = - 4 \int_{\Dinner} \Delta \bar g \, ( \bar g  -\data)  \, \mathrm{d} x
          = 0 \,. 
\end{equation*} 
We obtain $ \int_{\Dinner} \sabs{2 \nabla \bar g - \nabla g }^2 \, \mathrm{d} x 
= \int_{\Dinner} \abs{\nabla u(x,T)}^2 \, \mathrm{d} x$ and from \eqref{eq:ene2}, we deduce $\en_w(T) = \en_u(\cdot,T)$. 
We recall that $a$ is a positive function. From
$$
   \frac{\mathrm{d} \en_w}{\mathrm{d} t} 
       = 2\,\int_{\R^n} \left[ c^{-2}(x)\,[\partial_t w]\,[\partial_t^2 w] 
                 + [\nabla w]\,[\nabla \partial_t w]\right] \,\mathrm{d} x 
       = - 2\,\int_{\R^n} a\,\left( [\partial_t u]^2 -  [\partial_t v]^2 \right) \,\mathrm{d} x 
$$
together with 
\begin{align*}
   \en_u(\cdot,T) - \en_u(0) 
    & = - 2\,\int_0^T \int_{\R^n} a(x)\,[\partial u]^2(x,s)\,\,\mathrm{d} x \,\mathrm{d} s =: -C_u < 0
\\   \en_v(T) - \en_v(0)
& = 2\,\int_0^T \int_{\R^n} a(x)\,[\partial v]^2(x,s)\,\,\mathrm{d} x \,\mathrm{d} s =: C_v > 0
\end{align*}
we infer  $ \en_w(T) - \en_w(0) = \en_u(\cdot,T) - \en_u(0) + C_v \geq \en_u(\cdot,T) - \en_u(0)$.  With  $\en_w(T) = \en_u(\cdot,T)$ we obtain  $\en_u(0) \geq \en_w(0)$ and therefore 
\begin{multline}\label{intrelf}
       \int_{\R^n} \left[ c^{-2}(x)\,\ini_2^2(x) + \abs{ \nabla \ini_1(x) }^2 \right] \,\mathrm{d} x  \\
           \geq \int_{\R^n}  \left[ c^{-2}(x) \, \abs{ \ini_2(x) - \partial_t v  (x,0)}^2 
                         + \abs{\nabla  v(x,0) - \nabla \ini_1(x) } ^2 \right] \,\mathrm{d} x \,,
\end{multline}
where we have used the explicit expressions for $\en_u(0) $ and $\en_v(0)$, respectively.

With  $\f^* \coloneqq 2  \VVo_T   \UUo_T \fini $ we have   
$\fixT_2 \fini = \fini - \f^*$.  Moreover, writing  $v_0 \coloneqq  v(\cdot,0)|_{\Dinner}$ we have  
$f^* = \Po (v_0)$  and thus  $\Delta [ v_0 - f^*] = 0$ in  $\Dinner$. From this we infer 
$ \int_{\Dinner}  \inner{   \nabla v_0 -  \nabla \ini_1^* }{ \nabla \ini_1^* - \nabla \ini_1 }  \, \mathrm{d} x 
     = -\int_{\Dinner}  \inner{ \Delta v_0 - \Delta  f^* }{ f^*  - f }  \, \mathrm{d} x  
     = 0$ and thus 
\begin{multline} \label{eq:est2}
     \int_{\Dinner}  \abs{ \nabla v_0 - \nabla \ini_1 } ^2  \, \mathrm{d} x
     \\= \int_{\Dinner}  \abs{ \nabla v_0 - \nabla \ini_1^* } ^2 \, \mathrm{d} x
     + 
     \int_{\Dinner}  \abs{ \nabla \ini_1^* - \nabla \ini_1 } ^2 \, \mathrm{d} x
     \geq  \int_{\Dinner}  \abs{ \nabla \ini_1^* - \nabla \ini_1 } ^2 \, \mathrm{d} x.
 \end{multline}
Together  with~(\ref{intrelf}) this implies 
\begin{equation*}
       \int_{\R^n} \left[ c^{-2}\,\ini_2^2 + \abs{ \nabla \ini_1 }^2 \right] \,\mathrm{d} x 
           \geq  \int_{\R^n}  c^{-2}(x) \, \abs{ \ini_2(x) - \partial_t v  (x,0)}^2   \mathrm{d} x + \int_{\Dinner}  \abs{ \nabla \ini_1^* - \nabla \ini_1 } ^2 \, \mathrm{d} x \,,
\end{equation*}
which is $\norm{\fini}_{\Hini} \geq \snorm{ \fixT_2 f }_{\Hini}$.

Next we  show that the strict inequality holds. To that end assume the above equality holds and show that 
$f$ must vanish on $\R^n$. From (\ref{intrelf}) and \eqref{eq:est2} we obtain
\begin{equation*}
       \int_{\Dinner}  \abs{ \nabla v_0 - \nabla \ini_1 } ^2  \, \mathrm{d} x 
        \geq  \int_{\R^n}  \left[ c^{-2}(x) \, \abs{ \ini_2(x) - \partial_t v  (x,0)}^2 
                         + \abs{\nabla  v(x,0) - \nabla \ini_1(x) } ^2 \right] \,\mathrm{d} x \,,
\end{equation*}
 and therefore 
 $$ \int_{\R^n}  c^{-2}(x)\,|\ini_2(x) - \partial_t v(x,0)|^2 \,\mathrm{d} x  
       + \int_{\Dinner^c}  |\nabla v(x,0)|^2 \,\mathrm{d} x  =0 \,.$$
In particular, $\ini_2 - \partial_t v (\cdot,0)$  vanishes on $\R^n$ and $\nabla v(\cdot,0)$ vanishes on $\Dinner^c$. 
Because $v(x,0)$ vanishes for $x \in  \Dinner^{(2)}_{2T}$, it follows that  $v(\cdot,0)$ vanishes 
on $\Dinner^c$.  Applying Lemma~\ref{lem:uni} for $u(\cdot,t) \coloneqq v(\cdot,T-t)$ yields 
$2 \bar g = v(\cdot,T) = u(\cdot,0) = 0$ on $\R^n$. In particular, $\UUo_T \fini = 0$ on $\Dinner^c$. 
From Theorem~\ref{T:uni}, we infer $f=0$  on $\R^n$, which concludes the proof. \\

\ref{timerev2}: Let us first consider the case $\lambda=1$.  We have to show that there exists a  constant $L < 1$ such 
that $\snorm{ \id - \VVo_T \UUo_T)   } \leq L $. To that end, let $\fini = (\ini_1, \ini_2) \in H^1_0(\Dinner)$, 
$u$ solve \eqref{eq:wave1}-\eqref{eq:wave3} with initial data $(\ini_1, \ini_2)$ and  $v$ solve  \eqref{eq:timerev1}-\eqref{eq:timerev3} with initial data 
$(h := \Eo \UUo_T \fini, 0) $. The error term $\e \coloneqq u - v$ satisfies the wave equation  $c^{-2}\,\partial_t^2 w = - \Delta w - a\,\partial_t (u+v)$ on $\R^n\times (0,T)$ and its energy at time $T$  is given by    
\begin{align*}
     \en_\e(T)  
        &=  \int_{\R^n}  \bigl[  c^{-2}(x)\,|\partial_t \e (x,T)|^2 + |\nabla \e \,(x,T)|^2 \bigr] \,\mathrm{d} x
     \\ &= \int_{\R^n}  c^{-2}(x) \abs{ \partial_t u (x,T)}^2 \,\mathrm{d} x 
     + \int_{\R^n} \abs{\nabla \bar{g}(x) - \nabla g (x) }^2  \,\mathrm{d} x \,.
\end{align*}
Here for the second equality we used $v(\cdot,T) = \bar g \coloneqq \Eo g$ and 
$\partial_t u(\cdot ,T) = 0$ and the abbreviation $g = u( \cdot ,T)$. 

The second term  in the above equation displayed satisfies  
\begin{multline*}
\int_{\R^n}  \abs{\nabla \bar{g}(x) - \nabla g (x) }^2  \,\mathrm{d} x =
\\
     \begin{aligned}
      &=      \int_{\R^n}  \inner{\nabla (g(x) - \bar{g} (x)) }{ \nabla (g(x) - \bar{g} (x)) } \,\mathrm{d} x 
     \\ 
     &=  \int_{\R^n}  \inner{\nabla (g(x) - \bar{g} (x)) }{ \nabla (g(x) + \bar{g} (x)) } \,\mathrm{d} x  
     - 2 \int_{\R^n}  \inner{\nabla (g(x) - \bar{g} (x)) }{ \nabla \bar{g} (x)  } \,\mathrm{d} x 
     \\ 
     &=  \int_{\Dinner}  \abs{\nabla g (x)}^2  \,\mathrm{d} x 
          - \int_{\Dinner} \abs{ \nabla \bar g(x)}^2 \,\mathrm{d} x  
          + 2 \int_{\R^n}  \bigl( g(x) - \bar{g} (x) \bigr) \, \Delta \bar{g} (x)  \,\mathrm{d} x 
       \\ 
       &= \int_{\Dinner}  \abs{\nabla g (x)}^2  \,\mathrm{d} x 
         - \int_{\Dinner} \abs{\nabla \bar g(x)}^2  \,\mathrm{d} x \,.
\end{aligned}
\end{multline*}
As a consequence, we obtain 
     \begin{align*}
     \en_w(T) 
        &= \int_{\R^n}  \bigl[ c^{-2}(x) \abs{ \partial_t u (x,T)}^2 
             + \abs{\nabla g(x) }^2 \bigr] \,\mathrm{d} x 
              - \int_{\Dinner^c}  \abs{\nabla g (x)}^2  \,\mathrm{d} x \, 
              \\&= 
              \int_{\R^n}  \bigl[ c^{-2}(x) \abs{\partial_t u (x,T)}^2 
               + \abs{\nabla u(x,T)}^2 \bigr] \,\mathrm{d} x 
               - \| \UUo_T \fini\|^2_{H_0^1(\Dinner^c)}  
               \\&= \en_u(\cdot,T)- \|\UUo_T \fini\|^2_{H_0^1(\Dinner^c)},
\end{align*}
which implies $\en_w(T) + \|\UUo_T \fini\|^2_{H_0^1(\Dinner^c)} = \en_u(\cdot,T)$.
Because $\en_\phi(T) = \en_\phi(0) - C_\phi$ with $C_\phi>0$ for $\phi\in\{w,u\}$ and 
$C_w - C_u = - 2\,\int_0^T \int_{\R^n} a\,[\partial_t v]^2\,\,\mathrm{d} x \,\mathrm{d} t\leq 0$, we obtain
$$
     \en_w(0) + \|\UUo_T \fini\|^2_{H_0^1(\Dinner^c)} 
        =    \en_u(0) + C_w - C_u 
        \leq \en_u(0)
        = \norm{\f}^2_{\Hini}.
$$
Using that $\en_\e(0) = \int_{\R^n}  \left[ c^{-2}(x)\,|\partial_t v(x,0)-\ini_2|^2 
+ |\nabla v(x,0) - \nabla \ini_1(x)|^2 \right] \,\mathrm{d} x$  and applying Theorem~\ref{thm:stability}, we obtain 
\begin{equation} \label{E:ineq} \int_{\R^n}  c^{-2}(x)\,|\partial_t v(x,0)-\ini_2(x)|^2  + \int_{\Dinner} \Big|\nabla v(x,0) - \nabla \ini_1(x)\Big|^2 \,\mathrm{d} x
  \leq \left( 1-\frac{1}{C^2} \right) \norm{\f}^2_{\Hini}.\end{equation}
The left hand side in the above equation can be estimated as  
\begin{align*}  
    \int_{\Dinner} \sabs{ \nabla v(x,0) - \nabla \ini_1(x)}^2 \,\mathrm{d} x  
       & = \int_{\Dinner} \sabs{ \nabla (v(\cdot,0) - \Po (v(\cdot,0)) ) 
                + \nabla (\Po (v(\cdot,0)) - f) } ^2 \,\mathrm{d} x \\ 
       & =\int_{\Dinner} \sabs{ \nabla (v(\cdot,0) - \Po (v(\cdot,0)))}^2 
                + \sabs{ \nabla (\Po (v(\cdot,0)) - f) }^2 \,\mathrm{d} x \\& 
\geq \norm{\Po (v(\cdot,0))  - f}^2_{H_0^1(\Dinner)},
\end{align*} 
where we have used the fact that $\int_{\Dinner}   \inner{\nabla v(\cdot,0) 
- \nabla \Po (v(\cdot,0))}{ \nabla \Po (v(\cdot,0)) - \nabla  f}\,\mathrm{d} x 
= \int_{\Dinner} \Delta [v(\cdot,0) - \Po (v(\cdot,0))] \, (\Po (v(\cdot,0)) - f)\,\mathrm{d} x =0$. 
From \ref{E:ineq}, we arrive at 
\begin{equation*} 
    \| \Ko_1 \f\|^2_{H_0^1(\Dinner)} 
         =   \|\partial_t v(x,0)-\ini_2\|^2_{L^2(\Dinner,c^{-2})}  +  \|\Po (v(\cdot,0))  - f\|^2_{H_0^1(\Dinner)} 
       \leq   \left( 1-\frac{1}{C^2} \right) \norm{\f}^2_{\Hini} \,.
\end{equation*} 
This finishes the proof for the case $\lambda =1$. 

For  the general case note the identities   
\begin{equation*}
\fixT_\la  =  
\begin{cases} 
(1 - \la )\,\id + \la \fixT_1 & \text{  for } \lambda \in (0,1) 
\\
 (\la-1) \fixT_2 + (2-\la) \fixT_1 & \text{  for } \lambda \in (1,2) \,.
\end{cases}
\end{equation*}
Using the already verified estimates  $\snorm{\fixT_1} < 1$ and $\snorm{\fixT_2} \leq 1$, 
these equalities together with the  triangle inequality for the operator norm  show 
$\snorm{\fixT_\la} <  1$ for all $\lambda \in (0,2)$. 
\end{proof}

\begin{remark}[Iterative time reversal]  \label{rem:timerev} According to  Theorem \ref{thm:timerev} we have  $ \snorm{\fixT_\la}  = \snorm{\id -   \la  \VVo_T \UUo_T} < 1 $ on $\Hini$ for any  $\la \in (0,2)$. As a consequence, the Neumann  series  $\sum_{j \in \N} (\id - \la \VVo_T \UUo_T)^j$  in the operator norm converges to  $(\la \VVo_T \UUo_T)^{-1}$. This gives the inversion formula  
\begin{equation} \label{eq:neumann}
\fini = \sum_{j \in \N} (\id - \la \VVo_T \UUo_T)^j (\la \VVo_T \data) \quad \text{ for data  } \data =\UUo_T \fini \,.
\end{equation} 
For standard  PAT data, the Neumann series solution was first proposed in \cite{stefanov2009thermoacoustic} and further developed in \cite{stefanov2011thermoacoustic,tittelfitz2012thermoacoustic,homan,stefanov2015multiwave,nguyen2016dissipative,palacios2016reconstruction,katsnelson2018convergence,acosta2018thermoacoustic}. For full field PAT data with variable sound speed the  time method has been proposed and analyzed  in \cite{time_reversal_ff}.              
\end{remark}

One disadvantage of the time reversal method is that it requires full knowledge of $u(\cdot,T)$ on $\Dinner^c$. For the partial data case, the use of iterative methods may result in better image reconstruction, specifically if combined with regularization techniques integrating suitable regularization. Note that iterative reconstruction methods for variable sound speed based on an adjoint wave equation have been studied in \cite{huang2013full,belhachmi2016direct,arridge2016adjoint,haltmeier2017analysis,javaherian2018continuous,huang2013full,poudel2019survey}. Uniqueness and stability for standard PAT were studied in \cite{xu2004reconstructions,hristova2008reconstruction,stefanov2009thermoacoustic,stefanov2011thermoacoustic,nguyen2011singularities}, just to name a few.

\section{Numerical inversion}

For the numerical simulations, we consider the complete full-field PAT problem, which also involves the X-ray projection of the pressure wave (see Section~\ref{sec:ffp}). As described in more detail below, we consider a 2D version of full-field PAT and use iterative and variational regularization methods for its inversion.

\subsection{Problem description}

We consider a 2D version of the full-field PAT problem that results in the case of translational symmetry, which means that the sound speed, the damping, and the initial source are independent of the third argument. We assume the 2D object to be contained in the disc $\Dinner = \{(x_1, x_2) \in \mathbb{R}^2 \mid \| (x_1, x_2) \| < 1 \}$ and that the initial data has the form $\fini = (\ini, -c^2 a \ini)$. We then study the numerical solution of
\begin{equation} \label{eq:2d}
	\data = \Xo \Wo \ini + \xi
\end{equation}
where $\xi$ is the noise, $\Wo \ini = u(\cdot, T)|_{D^c}$ is the solution of \eqref{eq:wave1}-\eqref{eq:wave3} with initial data $(\ini, -c^2 a \ini)$ at time $T$, and $\Xo h (\theta, \xi) = \int_{\R} h(\xi \theta + s \theta^\bot) \, ds$ is the exterior Radon transform for $s \in \R \setminus [-1,1]$ and $\theta \in \sph^1_+$ where $\sph^1_+ = \{(\theta_1, \theta_2) \in \mathbb{R}^2 \mid \| (\theta_1, \theta_2) \| = 1 \wedge \theta_1 > 0\}$, the semicircle covering \SI{180}{\degree}.
Note that  $\Wo$ is pseudo-differential operator of order zero (see for example \cite{palaciospartial}) and therefore is a bounded mapping between Sobolev spaces of the same order.

Recovering $\ini$ from data \eqref{eq:2d} will be referred to as the full data case. We will also study a limited angle situation where we restrict the angular range to a proper subset $I \subsetneq \sph^1_+$ modeled by multiplication with the indicator function $\Mo_I$. Figure \ref{fig:data} shows the sound speed $c$, the attenuation coefficient $a$, the initial pressure $\ini$, the forward data $\Wo \ini$, the complete data $\Xo \Wo \ini$, and the limited angle data $\Mo_I \Xo \Wo \ini$, both with noise added.

\begin{figure}[htb!]
\begin{center}
    \includegraphics[width=\textwidth, height=0.5\textwidth]{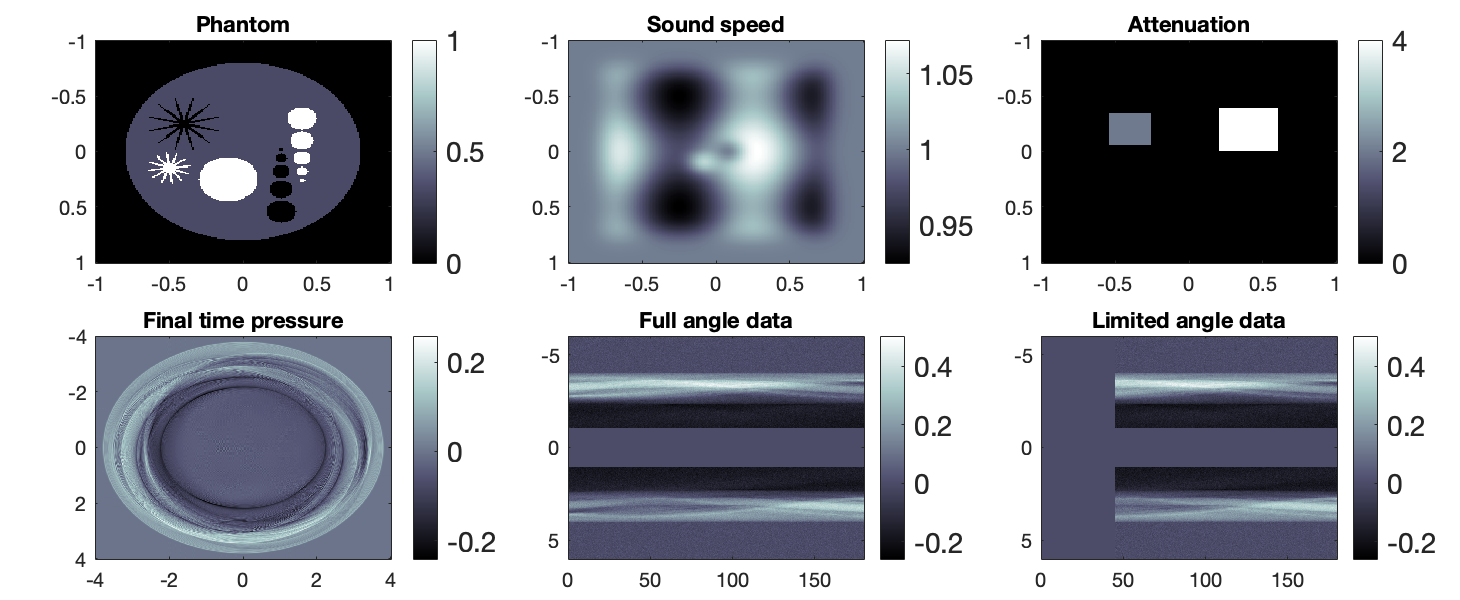} 
\caption{Simulation setup: Phantom $\ini$ (top left), sound speed $c$ (top middle),  attenuation $a$ (top right), final time pressure $\Wo \ini = p(\cdot, T)$ (bottom left), full angle data $\Xo \Wo \ini$ (bottom middle), limited angle data $\Mo_I \Xo \Wo \ini$ (bottom right).
\label{fig:data}
}
\end{center}  
\end{figure}

\subsection{Regularized inversion}

In the theoretical part, we have mathematically proven the uniqueness and stability of the final time wave inversion. Inversion of the Radon transform, however, is ill-posed and requires regularization methods. This is already important for the full angle case and even more important for the limited view case. As the inner well-posed wave inversion also requires an iterative solution, we follow the one-step reconstruction proposed in \cite{haltmeier2019}, where we recover $\ini$ directly from data \eqref{eq:2d} without inverting the X-ray transform as an intermediate step.

Probably the most established regularization methods are variational regularization \cite{engl1996,scherzer2009variational} and iterative regularization \cite{bakushinsky2005iterative,iterativebook}, which we will both use in this work. For the sake of simplicity, we will work in a discrete setting where the unknown and the data are 2D discrete images $X, Y$ and the forward operator has the discrete representation $\Ad$. Details on the discretization are given in the following subsection.

\paragraph{Iterative  regularization:} 
The first class of methods that we use is iterative regularization, which uses iterative methods for minimizing $\norm{\Ad X - Y}^2/2 \to \min_X$, where the regularization effect is introduced by early stopping of the iteration. For the preparation of the numerical results, we use Landweber \cite{hanke1995convergence}, the steepest descent \cite{neubauer1995convergence}, and the CGNE (conjugate gradient for normal equation) method \cite{iterative2}. Landweber and the steepest descent method are gradient methods for $\norm{\Ad X - Y}^2/2$ and can be written as $X_{k+1} = X_{k} - \gamma_k \Ad^*(\Ad X_{k} - Y)$, where the step size is $\gamma_k = \norm{\Ad^*(\Ad X_{k} - \data)}^2/\norm{\Ad \Ad^*(\Ad X_{k} - \data)}^2$ for the steepest descent method and $\gamma_k = \operatorname{const}$ for Landweber's method. We compared them with the CGNE method, which is summarized in Algorithm~\ref{alg:cgne}.

\begin{algorithm}[htb!]
    \SetAlgoLined
    \KwIn{Initial guess $X_0$}
    \KwOut{Final iteration $X_{k}$}
    Initialize 
    $p_0 = Y -\Ad  X_0$,
    $d_0 = \Ad^*p_0$
    $k = 0$\;
    \While{stopping criteria not satisfied}{
        $\alpha_{k}=\norm{\Ad^*p_{k}}^2 / \norm{\Ad d_{k}}^2$\;
        $X_{k+1}=X_{k}+\alpha_k d_{k}$\;
        $p_{k+1} = p_{k}-\alpha_k \Ad^* p_{k}$\;
        $\beta_k = \norm{\Ad^*p_{k+1}}^2/\norm{\Ad^*p_{k}}^2$\;
        $d_{k} = \Ad^*p_{k+1}+\beta_k d_{k}$\;
        $k\gets k+1$\;
    }
    \caption{CGNE algorithm for minimizing $ \norm{ \Ad X  - Y}^2/2$.} \label{alg:cgne}
\end{algorithm}

Together with suitable stopping criteria, the Landweber, the steepest descent, and the CGNE are known to be regularization methods \cite{iterativebook}. All methods behaved similarly, and CGNE turned out to be slightly faster, which has therefore been selected for the results.

\paragraph{Variational regularization:} As an alternative to iterative methods, variational regularization minimizes the generalized Tikhonov functional 
\begin{equation} \label{eq:tikhonov}
\tik_{\lambda, \data}(\ini) \coloneqq 
\frac{1}{2}\norm{\Ad X - Y}^2 + \lambda \reg(X)
\end{equation} 
using a suitable regularization functional $\reg(X)$ where $\lambda$ is the regularization parameter. For minimizing \eqref{eq:tikhonov}, iterative algorithms are used. In this work, we use forward-backward splitting (FBS) algorithm \cite{fbs}  
\begin{equation} \label{eq:fbs}
	X_{k+1}=\prox_{s\lambda \reg}\bigl( X_{k} - s \Ad^*(\Ad X_{k} - Y) \bigr) \,,
\end{equation}
with step size $s$ and regularization parameter $\la$ that is especially useful when the proximal mapping $\prox_{s\lambda R}(X) = \argmin_h\norm{h-X}^2/2+s\lambda \reg(X)$ is  known analytically. Further, for minimizing \eqref{eq:tikhonov} with  the regularizer $\reg(X) = \norm{\Dd X}_1$ where $\Dd$ is the spatial gradient operator, we use the  Chambolle-Pock (CP) primal-dual algorithm \cite{chambolle2011first}. Specifically we adapt the version of \cite{sidky2012convex} to our situation which is summarized in Algorithm \ref{alg:cp}. Both FBS and CP algorithm are known to converge to the minimizer of \eqref{eq:tikhonov}.

\begin{algorithm}[htb!]
    \SetAlgoLined
    \textbf{Initialization:}\\
    $L \leftarrow \norm{(\Ad,\Dd)}_2$, $\tau = 1/L$, $\sigma = 1/L$, $\theta = 1$\;
    $k\leftarrow 0$, $X_0, p_0, q_0 \gets 0$, $u_0\gets X_0$\;
    \While{stopping criteria not satisfied}{
        $p_{k+1} \leftarrow (p_{k}+\sigma(\Ad u_{k}-Y))/(1+\sigma)$\;
        $q_{k+1}\leftarrow \lambda(q_{k}+\sigma \Dd u_{k})/\text{max}\{\lambda\textbf{1},|q_{k}+\sigma \Dd u_{k}|\}$\;
        $X_{k+1}\leftarrow X_{k}-\tau \Ad^* p_{k+1}+\tau \Dd^* q_{k+1}$\;
        $u_{k+1}\leftarrow X_{k+1}+\theta(X_{k+1}-X_{k})$\;
        $k\leftarrow k+1$\;
    }
      \caption{CP algorithm for minimizing $ \norm{ \Ad X  - Y}^2/2 + \lambda \norm{\Dd X}_1$.} \label{alg:cp}
      \end{algorithm}

For all algorithms, we choose a weighted norm on the $Y$-data space  that accounts for the smoothing of the forward map $\Ad$ by a degree of $1/2$. Specifically, we use a discretization of the $\Lambda$-operator  defined by $\mathcal{F}_2 \Lambda \Phi (\theta, \omega)  = (\abs{\omega}/(4\pi)) (\mathcal{F}_2\Phi)(\theta, \omega)$ where $\mathcal{F}_2$ is the Fourier transform in the second variable. The adjoint $\Ad^*$ is then taken with respect to the weighted inner product, which is a discretization of the filtered backprojection (FBP) inversion formula for the full data Radon transform \cite{natterer2001mathematics}.

\subsection{Adjoint operator}
\label{sec:adjoint}

The iterative and variational algorithms  we described  above  requite the  adjoint of the forward map. For that purpose we use  a discretization of the continuous adjoint of the  initial-to-final time operator in $L^2$ space which is computed in this subsection.

\begin{lemma}[$L^2$ adjoint]
Let us consider the time reversed  wave equation   
\begin{align}  
\label{eq:rwave1}
&[c^{-2}(x)\partial_{tt} - a(x)\partial_t - \Delta] u(t,x)
= 0 \quad (x,t)\in\R^d\times (0,\infty)
\\ \label{eq:rwave2}
& u(x,T)   = g(x) \quad x \in\R^d
\\ \label{eq:rwave3}
& \partial_t u(x,0) = c^2(x) a(x) g(x)  \quad x \in \R^d \,.
\end{align}
Then $(\Wo^* g)(x) = \chi_{\Dinner} q(\cdot,0)$.
\end{lemma}

\begin{proof} Let $v$ be the solution of \eqref{eq:rwave1}- \eqref{eq:rwave3} and write   $v(x,t) \coloneqq \int_0^t u(x,\tau) d\tau$. The $\partial_t v = u $ and $v$ satisfies 
\begin{align*}  
&[c^{-2}(x)\partial_{tt} - a(x)\partial_t - \Delta] u(t,x)
= 0 \quad (x,t)\in\R^d\times (0,\infty)
\\ 
& u(x,T)   = 0 \quad x \in\R^d
\\ 
& \partial_t u(x,0) =  g(x)  \quad x \in \R^d \,.
\end{align*}
We have
\begin{align*}
0 &= \int_{\R^d} \int_0^T (c^{-2}(x)v_{tt}(x,t) +a(x)\partial_t v(x,t)-\Delta v(x,t))q(x,t) dt dx \\
 &= \int_{\R^d} c^{-2}(x)(\partial_t v(x,t) q(x,t)-v(x,t)q_t(x,t))|_0^T dx 
 \\ 
 & \qquad\qquad +\int_{\R^d}\int_0^T c^{-2}(x) v(x,t)q_{tt}(x,t) dt dx 
   + \int_{\R^d} a(x)v(x,t)q(x,t)dx |_0^T
   \\ 
&\qquad\qquad - \int_{\R^d}\int_0^T a(x)v(x,t)q_t(x,t) dt dx 
-\int_{\R^d}\int_0^T v(x,t)\Delta q(x,t) dt dx \\
&  = \int_{\R^d} c^{-2}(x)[\partial_t v(x,T) q(x,T) - v(x,T)q_t(x,T) - \partial_t v(x,0) q(x,0)+ v(x,0) q_t(x,0)]dx \\
& \qquad\qquad +\int_{\R^d} a(x)[v(x,T)q(x,T)-v(x,0)q(x,0)]dx 
\\ & \qquad\qquad + \int_{\R^d}\int_0^T [c^{-2}(x) q_{tt}(x,t)-a(x)q_t(x,t)-\Delta q(x,t)]v(x,t)dt dx \\
&=\int_{\R^d} c^{-2}(x)[u(x,T) g(x) - v(x,T)c^2(x) a(x)g(x) 
\\ & \qquad\qquad - f(x) q(x,0)]dx+ \int_{\R^d} a(x)v(x,T)g(x)dx.
\end{align*}
Therefore  $\int_{\R^d} c^{-2}(x)[(W_Tf)(x) g(x)- f(x) q(x,0)]dx =0,$ which is $ \Wo^*g  = \chi_{\Dinner} q(\cdot,0)$
\end{proof}

\subsection{Numerical details}
\label{sec:details}
        
For the implementation, all involved functions and operators are discretized and implemented in Matlab. We discretize the photoacoustic source, sound speed, and attenuation coefficient on a regular $201 \times 201$ Cartesian grid on the domain $[-1,1]^2$ and the solution of the wave equation up to time $T$ on an $801 \times 801$ Cartesian grid covering the domain $[-4,4]^2$. We use 600 discrete time steps in the time interval $[0,3]$. All algorithms require the forward map $\Ad$ as the discretization of $\Mo_I \Xo \Wo$ and $\Ad^*$ as the discretization of the adjoint map $\Wo^* \Xo^* \Mo_I$. This is realized by discretizing each factor as described next.

\paragraph{Wave operator:} For computing $\Wo$ and its adjoint, we numerically solve the forward damped wave equation \eqref{eq:wave1}-\eqref{eq:wave3} and the adjoint equation \eqref{eq:rwave1}-\eqref{eq:rwave3} with a variant of the k-space method described in \cite[Appendix A.2]{haltmeier2019reconstruction}. The k-space method is based on a Fourier expansion in the spatial variable \cite{kspace2, kspace3} and gives periodic solutions. The dimensions have been selected such that neither the forward solution on $[-4,4]^2$ nor the adjoint solution in $\Dinner$ are affected by the periodic extension.

\paragraph{Radon transform:} The Radon transform $\Xo$ of the $801 \times 801$ final time  pressure is computed for 1000 directions evenly distributed in $[\SI{0}{\degree},\SI{180}{\degree}]$ and evaluated with the Matlab built-in function \texttt{radon.m}. The adjoint is computed with the built-in function \texttt{iradon.m} using the Ram-Lak filter accounting for the preconditioning mentioned above.

\paragraph{Data restriction:}
In the full angular case, we require the Radon transform $\Xo h(\theta, s)$ for values $\abs{s} \geq 1$. This will be realized simply by a mask taking the values either 0 (if $\abs{s} < 1$) or 1 (if $\abs{s} \geq 1$). The same strategy is followed for including a limited angular range where the mask has value one for angles in the available angular range. For the numerical simulations below, we use a missing angular range of \SI{45}{\degree}. The effect of this can be seen in the last two images in Figure~\ref{fig:data}.

\begin{figure}[htb!]
{}\hspace{0.14\textwidth} CGNE \hspace{0.25\textwidth} FBS \hspace{0.27\textwidth} CP  \\
\includegraphics[width=\textwidth, height=0.5\textwidth]{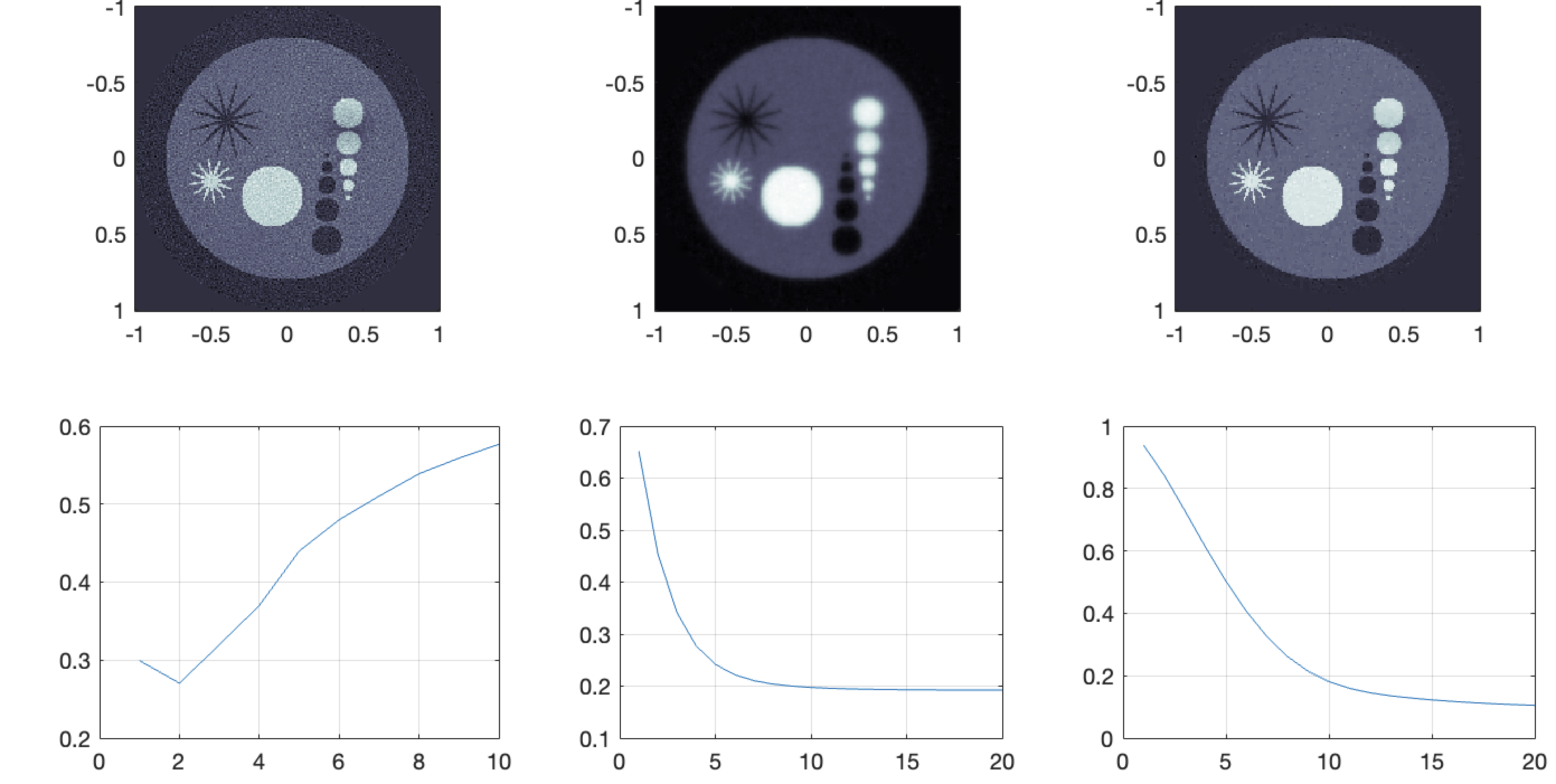} 
\caption{Results for full angular data where the top row shows the optimal reconstructions and the bottom row the relative reconstruction errors depending on the iteration index. Left: CGNE iterative regularization. Middle: FBS for quadratic regularization. Right: CP for TV-regularization.} \label{fig:full}
\end{figure}

\subsection{Numerical results}

Next, we present some results of our numerical simulations. We consider both the full angular range as well as the limitedangle case. As the ill-posedness of inverting $\Mo_I \Xo \Wo$ is governed by the ill-posedness of $\Mo_I \Xo$, the problem is only mildly ill-posed in the full angular case but severely ill-posed in the limited angle case. The data has been numerically computed using the discrete forward maps outlined above, and we additionally added Gaussian white noise with a standard deviation of \SI{0.5}{\percent} of the mean of the simulated data over all data pixels.

\paragraph{Full angular range:} 
We first use the complete angular data shown in the bottom middle image in Figure~\ref{fig:data}. All iterative regularization methods performed similarly, and we selected the CGNE for the presented results, which was slightly faster than the Landweber and the steepest descent method. Besides that, we show results using FBS with a quadratic regularizer and the CP algorithm for TV regularization. Reconstruction results and the relative $\ell^2$ reconstruction error as a function of the iteration index for all three methods are shown in Figure~\ref{fig:full}. In any case, the reconstruction with minimal relative reconstruction error is shown, which is $0.27$, $0.19$, and $0.11$ for CGNE, quadratic regularization, and TV regularization, respectively.

\begin{figure}[htb!]
{}\hspace{0.14\textwidth} CGNE \hspace{0.25\textwidth} FBS \hspace{0.27\textwidth} CP  \\
\includegraphics[width=\textwidth, height=0.5\textwidth]{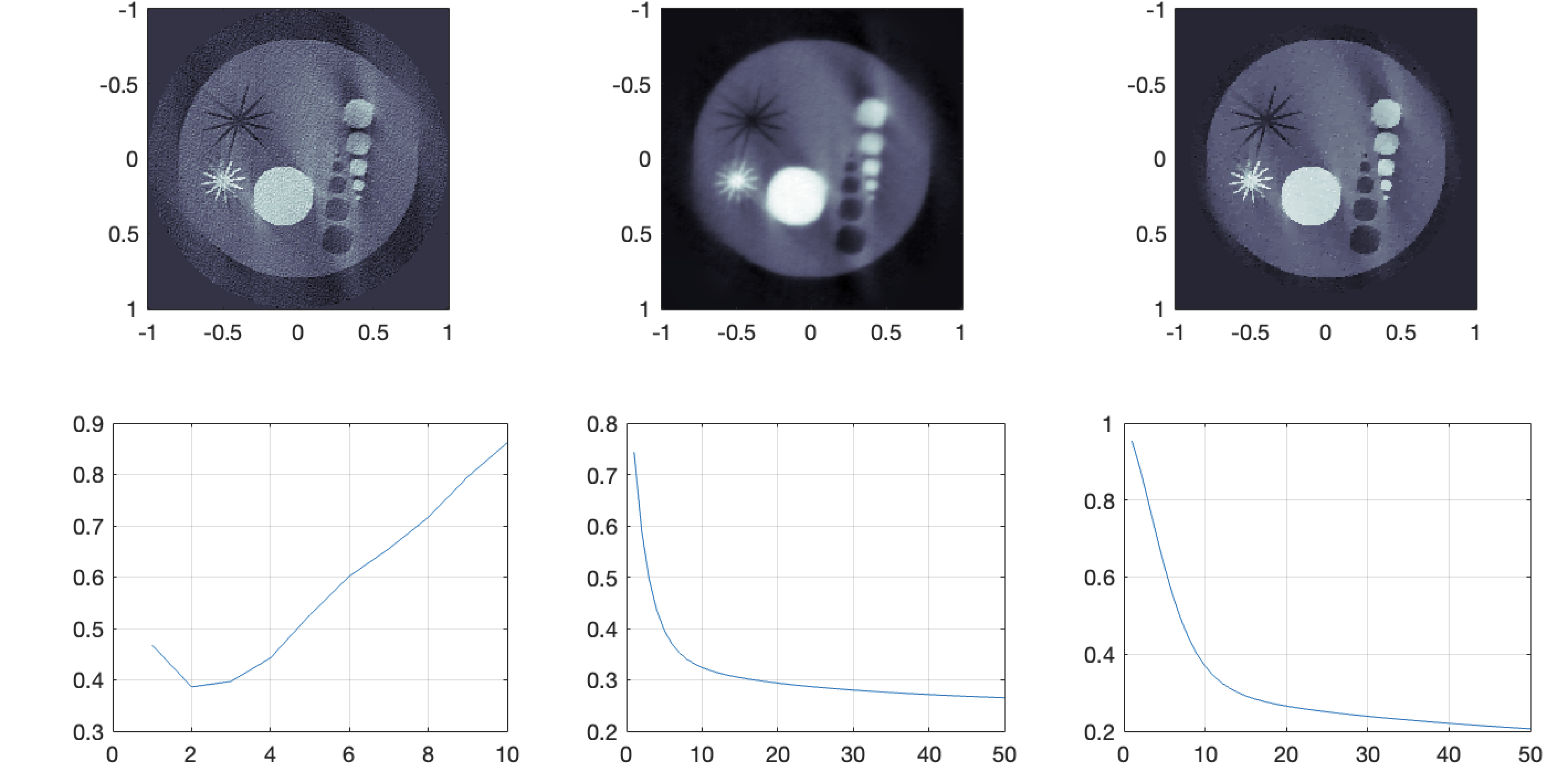}
\caption{Results for limited angular data where the top row shows the optimal reconstructions and the bottom row the relative reconstruction errors depending on the iteration index. Left: CGNE iterative regularization. Middle: FBS for quadratic regularization. Right: CP for TV-regularization} \label{fig:limited}
\end{figure}

\paragraph{Limited view data:} As another set of results, we restrict the angular range to $[\SI{45}{\degree}, \SI{180}{\degree}]$. We again  present results for  CGNE, FBS with a quadratic regularizer, and the CP algorithm for TV regularization. Reconstruction results and the relative $\ell^2$ reconstruction error as a function of the iteration index for all three methods are shown in Figure~\ref{fig:limited}. In any case, the reconstruction with minimal relative reconstruction error is shown, which is $0.39$, $0.27$, and $0.21$ for CGNE, quadratic regularization, and TV regularization, respectively.

\paragraph{Discussion:} 
In all our tests, the CGNE, as well as other iterative regularization techniques, turned out to be semi-convergent, reflecting the ill-posedness of the full-field PAT problem due to the involved Radon transform. Thus, the error decreases until a certain iteration and then starts increasing. Early stopping of iterations, which also means imposing regularization on the reconstruction process, yields a stable approximate solution. As another regularization technique, we use variational methods with either a quadratic regularizer $\reg(X) = \norm{\Dd(X)}^2_2$ or the TV penalty $\reg(X) = \norm{\Dd(X)}_1$. For the quadratic regularizer, we use forward-backward splitting, and for the TV penalty, we use the Chambolle-Pock primal-dual algorithm. As can be seen from Figures~\ref{fig:full} and \ref{fig:limited}, all methods are stable and convergent for both scenarios. In terms of approximating the ground truth image, the TV method yields fewer artifacts and a smaller reconstruction error. The reconstruction errors are sharp near the boundary, which reflects the choice of the TV term.

\section{Conclusions}
\label{sec:conclusion}

In this paper, we considered the full-field attenuated PAT problem with variable sound speed and variable damping. The associated inverse problem is a combination of an initial-to-final wave inversion problem and the X-ray transform. The latter is theoretically already well investigated, which is not the case for the wave inversion part. In this paper, we established the uniqueness and stability of wave inversion, showing that the ill-posedness is guided by the X-ray transform part. From the practical side, we used iterative and variational regularization techniques, including conjugate gradient, forward-backward splitting, and the CP method. All methods produced quite accurate results. Among these, in both full-angle and limited-angle data contaminated with noise, the Chambolle-Pock method applied to the TV regularization term performed the best. An interesting line of future work is the application of learning techniques, from which we expect further improvement. Another aspect we will work on is the application to real-world data and specific limited data constraints there.


\end{document}